\documentclass[a4paper,11pt]{amsart}
\usepackage{amsmath,amsthm,amsfonts,amssymb,mathrsfs}
\usepackage{enumerate}
\usepackage[colorlinks=true, linkcolor=blue, citecolor=magenta, menucolor=black]{hyperref}
\usepackage{geometry}
\usepackage[demo]{graphicx}
\usepackage[up]{caption}
\usepackage{subcaption}
\usepackage{pgf,tikz}
\usepackage[toc,page]{appendix}
\usetikzlibrary{arrows}
\usetikzlibrary{patterns}
\usepackage{color}
\usepackage{tikz-cd}

\numberwithin{equation}{section}
\theoremstyle{plain}
\newtheorem{thm}{Theorem}[section]

\newtheorem{lem}[thm]{Lemma}
\newtheorem{prop}[thm]{Proposition}
\newtheorem{cor}[thm]{Corollary}

\theoremstyle{definition}

\newtheorem*{defn*}{Definition}

\newcommand{\C}{\mathbb{C}}

\newcommand{\Q}{\mathbb{Q}}
\newcommand{\R}{\mathbb{R}}

\newcommand{\Z}{\mathbb{Z}}

\newcommand{\cU}{\mathcal{U}}

\newcommand{\cZ}{\mathcal{Z}}

\newcommand{\tM}{{\widetilde{M}}}

\newcommand{\spa}{\operatorname{span}}

\newcommand{\ot}{\otimes}
\newcommand{\ovt}{\mathbin{\overline{\otimes}}}
\newcommand{\Aut}{\operatorname{Aut}}

\newcommand{\Ad}{\operatorname{Ad}}

\newcommand{\odd}{{\operatorname{odd}}}
\newcommand{\Ind}{\operatorname{Ind}} 
 
\newcommand{\id}{\operatorname{id}}
\newcommand{\GL}{\operatorname{GL}}

\newcommand{\SL}{\operatorname{SL}}

\newcommand{\diag}{\mathrm{diag}}

\newcommand{\rC}{\operatorname{C}}

\newcommand{\PSL}{\operatorname{PSL}}

\allowdisplaybreaks

\begin{document}

\title{Infinite characters of type II on $\SL_n(\Z)$}
\author{R\'emi Boutonnet}
\address{Institut de Math\'ematiques de Bordeaux \\ CNRS \\ Universit\'e de Bordeaux \\ 33405 Talence \\ FRANCE}
\email{remi.boutonnet@math.u-bordeaux.fr}
\thanks{Research supported by ANR grant AODynG, 19-CE40-0008}

\maketitle

\begin{abstract}
We construct uncountably many infinite characters of type II for $\SL_n(\Z)$, $n \geq 2$. \end{abstract}

\section{Introduction}

Since the work of Bekka \cite{Be06} it is known that the special linear groups $\PSL_n(\Z)$, $n \geq 3$, have no characters but the obvious ones. Recall that a character on a group $\Gamma$ is a positive definite function $\phi: \Gamma \to \C$ which is conjugation invariant, normalized so that $\phi(e) = 1$ and extremal for these properties. 
Bekka's result states that every character on $\Gamma = \PSL_n(\Z)$, $n \geq 3$, is either the Dirac function $\delta_e$, or is equal to $1$ on a finite index subgroup $\Lambda < \Gamma$, in which case it factors through a character of the finite quotient $\Gamma/\Lambda$. This result was generalized for other higher rank semi-simple lattices by Peterson \cite{Pe14}, see also \cite{BH19, BBHP20} for other proofs and results in this direction.

A classical argument based on the GNS construction shows that, alternatively, a character is of the form $\phi = \tau \circ \pi$, where $\pi: \Gamma \to \cU(M)$ is a generating unitary representation into a von Neumann factor $M$ with a faithful normal finite trace $\tau$. Here generating means that $\pi(\Gamma)$ generates $M$ as a von Neumann algebra. Bekka's result can be rephrased by saying that the only generating unitary representation of $\PSL_n(\Z)$, $n \geq 3$, into a II$_1$-factor is the regular representation.

More recently, in an attempt to understand the unitary representations of these groups, he studied the generating representations into type I factors, i.e. the irreducible unitary representations of $\Gamma = \SL_n(\Z)$ (among other groups), \cite{Be18}. His objective was to understand if these groups had ``infinite characters'', i.e. if they admit irreducible representations $\pi: \Gamma \to \cU(H)$ such that the C*-algebra $C^*(\pi(\Gamma))$ contains non-zero trace class operators. In this case, the semi-finite trace on $B(H)$ gives rise to an infinite tracial weight on the universal C*-algebra $C^*(\Gamma)$, which is not purely infinite. This can be called an ``infinite character of type I''. In \cite{Be18}, Bekka constructs an infinite family of pairwise inequivalent such infinite characters for $\SL_n(\Z)$, $n \geq 3$, and for other groups like $\GL_n(\Q)$. This answered a question of Rosenberg \cite{Ro89}.
Motivated by his rigidity result for finite characters of type II, he asked if these groups can admit non-trivial infinite characters of type II, i.e. generating representations $\pi: \Gamma \to \cU(M)$, where $M$ is a factor of type II$_\infty$ such that $C^*(\pi(\Gamma))$ contains non-zero elements with finite trace in $M$, see \cite[Remark 5]{Be18}. Such a representation $\pi$ is said to be {\em traceable of type II$_\infty$}.

While his question was raised specifically for $\GL_n(\Q)$, which we are unable to treat at the moment, we do provide uncountably many characters of type II$_\infty$ for the linear groups $\SL_n(\Z)$, $n \geq 3$. Our approach follows the same ideas as in \cite{Be18}, except that we induce type II representations rather than finite dimensional ones.

The next proposition illustrates the main idea, even though the construction for $\SL_n(\Z)$ is a bit more elaborate.

Following Bekka-Kalantar \cite{BK19}, we say that a subgroup $\Lambda$ in a group $\Gamma$ is {\em a-normal} if $\Lambda \cap g\Lambda g^{-1}$ is amenable for every $g \in \Gamma \setminus \Lambda$.
For any group, we generically denote by $\lambda$ its left regular representation.

\begin{prop}\label{main prop}
Consider a group $\Gamma$ with a non-amenable, a-normal subgroup $\Lambda$. 
\begin{enumerate}
\item For every factorial representation $\pi_0: \Lambda \to \cU(H_0)$ which is not weakly contained in the regular representation, the induced representation $\pi := \Ind_{\Lambda}^{\Gamma}(\pi_0)$ is factorial and $\pi(\Gamma)''$ is naturally isomorphic with $B(\ell^2(\Gamma/\Lambda)) \ovt \pi_0(\Lambda)''$. Moreover, there exists a non-zero $x \in C^*(\pi(\Gamma))$ and a rank one projection $p_0 \in B(\ell^2(\Gamma/\Lambda))$ such that $x(p_0 \ot 1) = x$.
\item Given two representations $\pi_0,\pi_0'$ of $\Lambda$, denote by $\pi$ and $\pi'$, respectively, the induced $\Gamma$-representations. If $\pi$ is weakly contained in $\pi'$ then $\pi_0$ is weakly contained on $\pi_0' \oplus \lambda$. If $\pi_0$ is factorial, this further implies that $\pi_0$ is weakly contained in $\pi_0'$ or in $\lambda$.
\end{enumerate}
\end{prop}

\begin{cor}\label{cor1}
If $\Gamma$ contains a proper a-normal non-amenable virtual free subgroup, then it admits uncountably many factorial representations of type II$_\infty$ which are traceable, none of which weakly contains any other. 
\end{cor}

In general it is not so easy to construct interesting a-normal subgroups in a given group. For example, $\GL_2(\Z)$ can be viewed as an a-normal subgroup of $\SL_3(\Z)$ via the top-left embedding, but for $n \geq 4$, we do not know if $\SL_n(\Z)$ admits a non-amenable a-normal proper subgroup at all. Nevertheless, the same ideas allow to prove that inducing factorial representations of certain products of copies of $\SL_2(\Z)$ in $\SL_n(\Z)$ will still give satisfactory factorial representations of $\SL_n(\Z)$.

\begin{thm}\label{cor2}
For every $n \geq 2$, $\SL_n(\Z)$ admits uncountably many factorial representations of type II$_\infty$ which are traceable, none of which weakly contains any other.
\end{thm}

Our construction is ad hoc. We don't know if a similar result holds for, say, co-compact lattices in higher rank semi-simple Lie groups.
We point out that a similar argument can be used to produce factorial representations of type III of $\SL_n(\Z)$.

\subsection*{Acknowledgement} 
We thank Cyril Houdayer for inspiring discussions on this question and Bachir Bekka for useful comments on an earlier version of this note.

\section{Preliminaries}

\subsection{Factorial representations and weak containment}

By definition, a unitary representation $\pi: \Gamma \to \cU(H)$ is called {\em factorial} if the von Neumann algebra $\pi(\Gamma)''$ is a factor. We sometimes extend this terminology to specify the von Neumann type of $\pi(\Gamma)''$.

\begin{lem}\label{factorial}
If $\pi: \Gamma \to \cU(H)$ is a factorial unitary representation, then every subrepresentation of $\pi$ is weakly equivalent to $\pi$. More generally, if $\pi$ is the direct sum $\pi_1 \oplus \dots \oplus \pi_n$ of finitely many factorial representations  $\pi_1, \dots, \pi_n$ of $\Gamma$, then every subrepresentation of $\pi$ weakly contains one of the $\pi_i$'s.
\end{lem}
\begin{proof}
Clearly the second statement implies the first one. Assume that $\pi = \pi_1 \oplus \dots \oplus \pi_n$, for finitely many factorial representations  $\pi_1, \dots, \pi_n$ of $\Gamma$. Denote by $p_1,\dots, p_n$ the orthogonal projections on each of these direct summands, so that $\pi_i = p_i\pi$ for $i = 1,\dots,n$. 

Consider a non-zero invariant subspace $K \subset H$ and denote by $p \in B(H)$ the orthogonal projection onto $K$. Denote by $z \in \cZ(\pi(\Gamma)'')$ the central support of $p \in \pi(\Gamma)'$. Note that for every $i = 1,\dots, n$, $zp_i$ is either $0$ or $p_i$, because $p_i\pi(\Gamma)'' = \pi_i(\Gamma)''$ is a factor. 

Choosing an index $i$ such that $zp_i \neq 0$, we thus get $zp_i = p_i$. 
Assume that $x \in \pi(\Gamma)''$ is such that $px = 0$. By definition of the central support, this implies that $zx = 0$, and further, $p_ix = p_izx = 0$. So $\pi_i = p_i\pi$ is weakly contained in $p\pi$.
\end{proof}

\begin{lem}\label{cut}
If $\pi: \Gamma \to \cU(H)$ is a unitary representation which is weakly contained in the direct sum of two representations $\pi_1 \oplus \pi_2$, then $\pi$ is the direct sum of two representations: one weakly contained in $\pi_1$ and one weakly contained in $\pi_2$. If moreover $\pi$ is factorial then it is weakly contained in $\pi_1$ or $\pi_2$.
\end{lem}
\begin{proof}
By assumption the map $\pi_1(g) \oplus \pi_2(g) \mapsto \pi(g)$ extends to a C*-morphism $C^*(\pi_1 \oplus \pi_2) \to C^*(\pi)$. By Arveson extension theorem, this morphism extends to a ucp map $\Phi: B(H_1 \oplus H_2) \to B(H)$. Denote by $p_1, p_2 \in B(H_1 \oplus H_2)$ the orthogonal projections onto $H_1,H_2$, respectively. By multiplicative domain considerations, $\Phi(p_1), \Phi(p_2) \in \pi(\Gamma)'$.
For $i = 1,2$, denote by $r_i \in \pi(\Gamma)'$ the support projection of $\Phi(p_i) \in \pi(\Gamma)'$.

{\bf Claim .} For $x$ in the universal C*-algebra $C^*(\Gamma)$, if $\pi_i(x) = 0$, then $r_i \pi(x) = 0$. In particular, $r_i\pi$ is weakly contained in $\pi_i$.

Indeed $\pi_i(x) = 0$ means that $p_i(\pi_1 \oplus \pi_2)(x) = 0$. Applying $\Phi$, and using multiplicative domain, we get that $\Phi(p_i)\pi(x) = 0$. This easily implies the claim, by definition of the support projection.

Since $p_1 + p_2 = 1$, we find that $\Phi(p_1), \Phi(p_2), r_1$ and $r_2$ all commute to each other. Moreover since $\Phi$ is a ucp map, we have $0 \leq \Phi(p_i) \leq 1$, showing that $r_i \geq \Phi(p_i)$, for $i = 1,2$. Therefore $r_1 + r_2 \geq \Phi(p_1) + \Phi(p_2) = 1$. So $1 - r_1 \leq r_2$. From the claim, it follows that $r_1 \pi \prec \pi_1$, while $(1-r_1)\pi \subset r_2\pi \prec \pi_2$. This gives the desired decomposition $\pi = r_1\pi \oplus (1-r_1)\pi$.

The moreover part follows from Lemma \ref{factorial}.
\end{proof}

\subsection{Induced representations}

In this section we are given two groups $\Lambda < \Gamma$ and a unitary representation $\pi_0$ of $\Lambda$.

Denote by $s: \Gamma/\Lambda \to \Gamma$ a section to the natural projection map, and by $c: \Gamma \times \Gamma/\Lambda \to \Lambda$ the cocycle given by the formula $c(g,x) = s(gx)^{-1}gs(x)$, for $g \in \Gamma$, $x \in \Gamma/\Lambda$.
By definition, the representation $\pi = \Ind_\Lambda^\Gamma(\pi_0)$ is defined on the Hilbert space $H = \ell^2(\Gamma/\Lambda) \ot H_0$ by the formula
\[\pi_g(\delta_x \ot \xi) = \delta_{gx} \ot (\pi_0)_{c(g,x)}\xi, \text{ for all } g \in \Gamma, x \in \Gamma/\Lambda, \xi \in H.\]

The following easy lemma is a special case of a result of Mackey \cite{Mac52}. It is given in this form in \cite[Proposition 9]{Be18}.

\begin{lem}\label{Mackey}
Given another subgroup $\Sigma < \Gamma$, denote by $S \subset \Gamma$ a system of representatives for the double coset space $\Sigma \backslash \Gamma / \Lambda$. For each $s \in S$, we denote by $\pi_s$ the representation of $s\Lambda s^{-1}$ given by $\pi_s(sgs^{-1}) = \pi_0(g)$, for all $g \in \Lambda$. Then the restriction of $\pi$ to $\Sigma$ is equivalent to the direct sum 
\[\bigoplus_{s \in S} \Ind_{s\Lambda s^{-1} \cap \Sigma}^\Sigma({\pi_s|}_{s\Lambda s^{-1} \cap \Sigma}).\]
\end{lem}

\subsection{Finite index considerations}

We will sometimes need to induce representations from normal finite index subgroups. In this context we prove the following stability result. 

\begin{lem}\label{going up}
Consider a group $\Gamma$ with a finite index normal subgroup $\Lambda < \Gamma$. Consider a factorial unitary representation $\pi: \Lambda \to \cU(H)$ and denote by $\rho$ the induced representation $\rho = \Ind_\Lambda^\Gamma(\pi)$.
\begin{enumerate}
\item Then $\rho$ is the direct sum of finitely many factorial representations. If $\pi$ is of type II$_1$, so is $\rho$. 
\item Consider another factorial representation $\pi'$ of $\Lambda$ and its induced representation $\rho'$. Take an automorphism $\alpha \in \Aut(\Gamma)$ such that $\alpha(\Lambda) = \Lambda$. If a subrepresentation of $\rho'$ is weakly contained in $\rho \circ \alpha$, then $\pi'$ is weakly contained in $\pi \circ \Ad_g \circ \alpha$ for some $g \in \Gamma$ (and $\Ad_g$ denotes the automorphism of $\Lambda$ given by $g$-conjugation inside $\Gamma$). 
\end{enumerate}
\end{lem}

\begin{proof}
(1) By definition $\rho$ is acting on $\ell^2(\Gamma/\Lambda) \ovt H$. Since $\Lambda$ is normal in $\Gamma$, its restriction to $\Lambda$ is given by
\[\rho_h(\delta_{g\Lambda} \ot \xi) = \delta_{g\Lambda} \ovt \pi_{g^{-1}hg}(\xi), \text{ for every } h \in \Lambda, g \in \Gamma, \xi \in H.\]
In other words, $\rho|_\Lambda$ is equivalent to $\bigoplus_{g\Lambda \in \Gamma/\Lambda} \pi \circ \Ad_{g^{-1}}$. Here we note that the equivalence class of $\pi \circ \Ad_{g^{-1}}$ does not depend on the choice of the representative $g \in \Gamma$ in the class $g\Lambda \in \Gamma/\Lambda$.

Consider the von Neumann algebras $N \subset M \subset \tM$ defined by $N = \rho(\Lambda)''$, $M = \rho(\Gamma)''$ and $\tM = B(\ell^2(\Gamma/\Lambda)) \ovt \pi(\Lambda)''$.
For every $a \in \Gamma/\Lambda$, denote by $p_a \in N' \cap \tM$ the orthogonal projection onto $\delta_a \ot H$. Then we see that $p_a N = \delta_a \ot \pi(\Lambda)'' = p_a\tM p_a$.

In particular, $p_a(N' \cap \tM)p_a = \C p_a$ for every $a \in \Gamma/\Lambda$. So $N'\cap \tM$ admits a finite partition of unity consisting of minimal projections; it must be finite dimensional. In particular $\cZ(M)$ is finite dimensional, which precisely means that $\rho$ is the direct sum of finitely many factorial representations.

Assume that $\pi(\Lambda)''$ is of type II$_1$. Then $\tM$ is also of type II$_1$, and hence $M$ is a tracial von Neumann algebra. Moreover it contains $N$, which is of type II by our description of $\rho|_\Lambda$. So $M$ has no type $I$ direct summand, which proves that it is of type II$_1$.

(2) Take a subrepresentation $\sigma$ of $\rho'$. By Lemma \ref{factorial}, we find that $\sigma|_\Lambda$ weakly contains a representation of the form $\pi' \circ \Ad_h$ for some $h \in \Gamma$. So if $\sigma$ is weakly contained in $\rho \circ \alpha$, then restricting to $\Lambda$, we find that $\pi' \circ \Ad_h$ is weakly contained in $\bigoplus_{g\Lambda \in \Gamma/\Lambda} \pi\circ \Ad_g \circ \alpha$. Since $\pi'$ is factorial, Lemma \ref{cut} shows that $\pi' \circ \Ad_h$ is weakly contained in some $\pi \circ \Ad_g \circ \alpha$. Hence $\pi'$ is weakly contained in $\pi \circ \Ad_{g'} \circ \alpha$, for $g' = g\alpha(h)^{-1}$.
\end{proof}

\subsection{Many representations of virtual free groups}

We record here some obvious facts about the abundance of representations of free groups. 

\begin{lem}
There exists $n \geq 2$ such that for every $k \geq n$, the free group on $k$ generators admits uncountably many factorial representations $\pi_i$, $i \in I$, of hyperfinite type II$_1$ such that $\pi_i$ is not weakly contained in $\pi_j \circ \alpha$ for any distinct indices $i, j \in I$ and any $\alpha \in \Aut(F_n)$.
\end{lem}
\begin{proof}
There are uncountably many pairwise non-isomorphic finitely generated simple groups. By Juschenko-Monod's theorem \cite{JM13} we can even choose these groups to be all amenable. So we may find some $n$ large enough so that $F_n$ has uncountably many non-isomorphic simple quotients $\Lambda_i$, $i \in I$, which are amenable (and infinite). The regular representation of each $\Lambda_i$ yields by composition a unitary representation $\pi_i$ of $F_n$. Since $\Lambda_i$ is simple infinite, it is ICC, so this representation is factorial, and by amenability it generates the hyperfinite II$_1$-factor. 

Note that the kernel of $\pi_i$ is precisely the kernel of the quotient map $p_i : F_n \to \Lambda_i$. Given two indices $i \neq j$ and an automorphism $\alpha \in \Aut(F_n)$, if $\pi_{i} \prec \pi_{j} \circ \alpha$, then the quotient map $p_i : F_n \to \Lambda_i$ factors through $p_j \circ \alpha$. By simplicity of $\Lambda_j$, the factorized map $\Lambda_j \to \Lambda_i$ must be an isomorphism. This gives $i = j$. 

If $k \geq n$, then any representation of $F_n$ gives a representation of $F_k$ by composition with the natural surjection $F_k \to F_n$.
\end{proof}

\begin{lem}\label{many rep}
Consider a finitely generated group $\Gamma$ containing a non-abelian free group of finite index. Then $\Gamma$ admits uncountably many unitary representations $\pi_i$, $i \in I$ of hyperfinite type II$_1$, such that $\pi_i$ is not weakly contained in $\pi_j \circ \alpha$ for any distinct indices $i, j \in I$ and any $\alpha \in \Aut(\Gamma)$.
\end{lem}
\begin{proof}
Consider a free subgroup $F$ of finite index in $\Gamma$. Taking $F$ smaller if necessary, we can assume that $F = F_k$, for $k$ large enough so that the previous lemma holds true. We can also assume that $F$ is normal inside $\Gamma$. In fact, since $\Gamma$ is finitely generated, it admits only finitely many subgroups of a given finite index. So we may assume that in fact $F$ is characteristic in $\Gamma$, i.e. invariant under every automorphism of $\Gamma$.
Then the result follows from combining the above lemma with Lemma \ref{going up}.
\end{proof}

We observe that a representation of a group into the hyperfinite II$_1$-factor is amenable in the sense of Bekka. So if the group is non-amenable, it is not weakly contained in the regular representation.

\section{Proof of the main results}

\subsection{General results}

\begin{proof}[Proof of Proposition \ref{main prop}]
(1) Take a non-amenable a-normal subgroup $\Lambda < \Gamma$ and a representation $\pi_0$ of $\Lambda$.
Denote by $\pi = \Ind_\Lambda^\Gamma(\pi_0)$ the induced representation. 

Denote by $I := (\Gamma/\Lambda ) \setminus \{\Lambda\}$. Then $\Gamma/\Lambda = \{\Lambda\} \sqcup I$ is a $\Lambda$-invariant partition, so that $H$ is the direct sum of two $\Lambda$-invariant subspaces, $H_1 = \delta_\Lambda \ot H_0$ and $H_2 = \ell^2(I) \ot H_0$. Denote by $\pi_1$ and $\pi_2$ the two $\Lambda$ representations obtained this way. 

We see that $\pi_1$ is canonically isomorphic with $\pi_0$ while Lemma \ref{Mackey} describes $\pi_2$ as a direct sum of representations $\Ind_{\Sigma}^\Lambda(\sigma)$, where $\Sigma < \Gamma$ is of the form $\Lambda \cap g\Lambda g^{-1}$ for some $g \in \Gamma \setminus \Lambda$, and $\sigma$ is a representation of $\Sigma$. Since $\Lambda$ is a-normal in $\Gamma$, each such $\Sigma$ is amenable and thus $\sigma$ is weakly contained in the regular representation. After inducing to $\Lambda$ and taking the sirect sum, we find that $\pi_2$ is weakly contained in the regular representation of $\Lambda$.

By definition, we have an inclusion $\pi(\Gamma)'' \subset B(\ell^2(\Gamma/\Lambda)) \ovt \pi_0(\Lambda)''$, and our goal is to prove that this is an equality. By taking commutants, we need to prove that $\pi(\Gamma)' \subset 1 \ot \pi_0(\Lambda)'$. Denote by $p \in B(H)$ the orthogonal projection onto $H_1 = \delta_{\Lambda} \ot H_0$. As explained above, $p \in \pi(\Lambda)'$.

{\bf Claim.} $pT(1-p) = 0$ for every $T \in \pi(\Lambda)'$.

Otherwise, by classical von Neumann algebra theory, we could find a nonzero partial isometry $u \in \pi(\Lambda)'$ such that $uu^* \leq p$ and $u^*u \leq 1- p$. Then $u$ implements a conjugation between a subrepresentation of $\pi_1 \simeq \pi_0$ and a subrepresentation of $\pi_2$. By Lemma \ref{factorial}, this implies that $\pi_0$ is weakly contained in $\pi_2$, and further, in the regular representation. This is excluded by assumption.

Fix $T \in \pi(\Gamma)'$. In particular, $T$ commutes with $\pi(\Lambda)$ so the claim implies that $p$ commutes with $T$: there exists $T_0 \in B(H_0)$ such that $T(\delta_{\Lambda} \ot \xi) = \delta_{\Lambda} \ot (T_0\xi)$ for every $\xi\in H_0$. Further, we observe that $T_0 \in \pi_0(\Lambda)'$. Hence for every $g \in \Gamma$:
\begin{align*}
T(\delta_{g\Lambda} \ot \xi) & = T\pi(g)(\delta_{\Lambda} \ot (\pi_0)_{c(g,\Lambda)^{-1}}\xi)\\
& = \pi(g)(\delta_{\Lambda} \ot T_0(\pi_0)_{c(g,\Lambda)^{-1}}\xi)\\
& = \pi(g) (\delta_\Lambda \ot (\pi_0)_{c(g,\Lambda)^{-1}}T_0\xi)\\
& = \delta_{g\Lambda} \ot T_0\xi.
\end{align*}
Therefore, $T = \id \ot T_0 \in 1 \ot \pi_0(\Lambda)'$, as desired.

For the moreover part, observe that $\pi_1$ is not weakly contained in $\pi_2$, as $\Lambda$-representations. Hence, there exists $a$ in the universal C*-algebra $C^*(\Lambda)$ such that $\pi_1(a) \neq 0$ while $\pi_2(a) = 0$. This implies that $\pi(a) = \pi_1(a) + \pi_2(a) = \pi_1(a) = \pi(a)p$ and indeed, $p = p_0 \ot 1$ for some rank one projection $p_0$. 

(2) Assume that $\pi$ is weakly contained in $\pi'$. Then this is also true for the restriction to $\Lambda$ of these representations. In particular, $\pi_0$ is weakly contained in $\pi'|_{\Lambda}$ which is weakly contained in $\pi_0' \oplus \lambda_\Lambda$, as we observed in the proof of (1). The factorial case follows from Lemma \ref{cut}.
\end{proof}

\begin{proof}[Proof of Corollary \ref{cor1}]
Assume that $\Gamma$ contains a proper a-normal non-amenable subgroup $\Lambda$ which is virtually free. Note that $\Lambda$ is necessarily of infinite index inside $\Gamma$.

Lemma \ref{many rep} provides us with an uncountable family of unitary representations $\pi_i$, $i \in I$, of $\Lambda$ which are all amenable, factorial of type II$_1$, and none of which is weakly contained in any other.
In particular no such $\pi_i$ is weakly contained in the regular representation.

We may then induce these representations to $\Gamma$ and Proposition \ref{main prop} gives that the representations $\rho_i$, $i \in I$ that we get are all factorial of type II$_\infty$, traceable, and none of them is weakly contained in any other.
\end{proof}

Corollary \ref{cor1} raises the question whether virtual free groups themselves admit many traceable representations of type II$_\infty$. As expected this is the case, as follows from the next lemma.

\begin{lem}
Consider a group $\Gamma$ and a finite index normal subgroup $\Lambda < \Gamma$. Assume that $\Lambda$ admits a non-amenable a-normal proper subgroup. Then $\Gamma$ admits a non-amenable a-normal proper subgroup $\Gamma_0$ such that $\Gamma_0 \cap \Lambda$ has finite index inside $\Gamma_0$.
\end{lem}
\begin{proof}
Take a non-amenable a-normal proper subgroup $\Lambda_0 < \Lambda$.
Consider the quotient group $F = \Gamma/\Lambda$, with projection map $p: \Gamma \to F$, and take a maximal subset $I \subset F$ such that there exist lifts $g_i \in \Gamma$, $i \in I$, for which $p(g_i) = i$ for every $i \in I$ and $\bigcap_{i \in I} g_i\Lambda_0g_i^{-1}$ is non amenable. Denote by $\Lambda_1$ this non-amenable subgroup.

{\bf Claim.} For every $g \in \Gamma$, either $g \Lambda_1 g^{-1} \cap \Lambda_1$ is amenable or $g$ normalizes $\Lambda_1$.

Assume that $g \Lambda_1 g^{-1} \cap \Lambda_1$ is non-amenable. Then by maximality of $I$, we must have that $p(g)I = I$. Then for every index $i \in I$, we find $j \in I$ such that $p(g)i = j$. This means that $gg_i\Lambda_0 = g_j\Lambda_0$, and hence $gg_i\Lambda_0g_i^{-1}g^{-1} = g_j\Lambda_0 g_j^{-1}$. Applying this observation for every $i \in I$ and intersecting over $I$ we find that indeed $g$ normalizes $\Lambda_1$, as claimed.

Note that $\Lambda_1$ is a-normal inside $\Lambda$, being an intersection of a-normal subgroups of $\Lambda$. So it is equal to its own normalizer inside $\Lambda$. Furthermore, since $\Lambda$ has finite index inside $\Gamma$, the normalizer $N_\Lambda(\Lambda_1)$ of $\Lambda_1$ inside $\Lambda$ has finite index in the normalizer $N_\Gamma(\Lambda_1)$ inside $\Gamma$. So we conclude that $\Lambda_1 = N_\Lambda(\Lambda_1)$ has finite index inside $\Lambda_2 := N_\Gamma(\Lambda_1)$.

Let us check that $\Lambda_2$ is a-normal in $\Gamma$. Take $g \in \Gamma$ such that $g\Lambda_2g^{-1} \cap \Lambda_2$ is non-amenable. Since $\Lambda_1$ has finite index inside $\Lambda_2$, we find that $g\Lambda_1g^{-1} \cap \Lambda_1$ is non-amenable as well. By the claim this implies that $g \in \Lambda_2$.
\end{proof}

\begin{cor}\label{virtual free cor}
If $\Gamma$ is virtually free, non-amenable, then it admits uncountably many factorial representations of type II$_\infty$ which are traceable, and none of which is weakly contained in any other.
\end{cor}

\subsection{The case of $\SL_n(\Z)$}

The case $n =2$ is a special case of Corollary \ref{virtual free cor}.

Fix $n \geq 3$ and denote by $\Gamma := \SL_n(\Z)$. Denote by $d$ the integer part of $n/2$, so that $n = 2d$ or $n = 2d+1$. Denote by $\Sigma < \Gamma$ the copy of $\SL_2(\Z)^d$ given by block diagonal matrices $\diag(A_1,\dots,A_d,1_{\odd})$, where $A_1,\dots, A_d \in \SL_2(\Z)$ and $1_{\odd}$ is the empty matrix if $n$ is even and equals the $1\times 1$-matrix with entry $1$ if $n$ is odd.

Denote by $e_1,\dots,e_n$ the canonical basis in $V = \R^n$. For the natural action of $\Gamma$ on $V$, $\Sigma$ fixes the planes $V_k := \spa(\{e_{2k}, e_{2k+1}\})$, for $k = 1, \dots, d$. It also fixes the space $V_\odd$, defined to be $\R e_n$ if $n$ is odd and $0$ otherwise. For every $k = 1,\dots,d$, denote by $\Sigma_k < \Sigma$ the set of elements which preserve $V_k$. Then $\Sigma_k$ is a copy of $\SL_2(\Z)$ and $\Sigma = \Sigma_1 \times \dots \times \Sigma_d$.

Although $\Sigma$ is not a-normal in $\Gamma$, the family of subgroups $\{\Sigma_i, , \, i = 1, \dots, d\}$ satisfies a property of this kind (up to a finite index normalizer). The next lemma specifies this property. The task will be to extend Proposition \ref{main prop} to this setting.

\begin{lem}\label{a-normal family}
The following facts are true :
\begin{enumerate}
\item $\Sigma$ has finite index in its normalizer $\Lambda := N_\Gamma(\Sigma)$;
\item $\Lambda$ coincides with the set of elements in $\Gamma$ which globally preserve the direct sum decomposition $V = V_1 \oplus \dots \oplus V_d \oplus V_\odd$. In other words it is the set of elements which permute the spaces $V_k$, $k = 1,\dots, d$.
\item For every $g \in \Gamma \setminus \Lambda$, there exists $1 \leq k \leq d$ such that $\Sigma_k \cap g\Sigma g^{-1}$ is amenable.
\item For every $g \in \Gamma \setminus \Lambda$, there exists $1 \leq k \leq d$ such that $\Sigma_k \cap g\Lambda g^{-1}$ is amenable.
\end{enumerate}
\end{lem}
\begin{proof}
(1) Note that $\Sigma$ has finite index inside the set of elements $g \in \Gamma$ such that $gV_i = V_i$ for all $i = 1,\dots,d$. So (2) is easily seen to imply (1).

(2) Denote by $\Lambda'$ the set of elements which permute the spaces $V_k$, $k = 1,\dots,d$. It is easy to see that elements of $\Lambda'$ normalize $\Sigma$. So $\Lambda' \subset \Lambda$. For the converse inclusion, it suffices to check that property (3) holds for every $g \in \Gamma \setminus \Lambda'$, since clearly the conclusion of (3) prevents $g$ to normalize $\Sigma$.

(3) Take $g \in \Gamma \setminus \Lambda'$. Then there exists $1 \leq k \leq d$ such that $V_k$ is not equal to some $gV_i$, $i = 1,\dots,d$. Let us prove that $\Sigma_0 := \Sigma_k \cap g\Sigma g^{-1}$ is amenable.

Denote by $W_k := \sum_{i \neq k} V_i$, and by $p$ the projection onto $V_k$ parallel to $W_k$ and $q = 1-p$ the projection onto $W_k$. These projections are $\Sigma_0$-equivariant.
Let us take some $1 \leq i \leq d$ such that $p(gV_i) \neq 0$. Since $g(V_i)$ is globally $\Sigma_0$-invariant, $p(gV_i)$ is $\Sigma_0$-invariant as well. If $p(gV_i)$ is one dimensional, then we have found a $\Sigma_0$-invariant line inside the plane $V_k$, proving that $\Sigma_0$ acts amenably on $V_k$. Since it acts trivially on $W_k$, this implies that $\Sigma_0$ is amenable.

Assume on the contrary that $p(gV_i) = V_k$. Then $p$ implements a conjugation between $gV_i$ and $V_k$. But by assumption, $gV_i \neq V_k$. So must also have $q(gV_i) \neq 0$. If $q|_{gV_i}$ is injective, then we find that $\Sigma_0$ acts trivially on $gV_i$. Otherwise the kernel of $q$ intersects $gV_i$ into a line, which is globally $\Sigma_0$-invariant. In both cases we find a $\Sigma_0$-invariant line in $gV_i$, hence in $V_k$, and we conclude again that $\Sigma_0$ is amenable. 

(4) follows obviously from (1) and (3).
\end{proof}

\begin{prop}\label{main prop 2}
Consider a type II$_1$ factorial representation $\sigma$ of $\Sigma$, whose restriction to each $\Sigma_k$, $k = 1,\dots,d$, is factorial and not weakly contained in the regular representation. The following facts hold true.
\begin{enumerate}
\item The induced representation $\rho = \Ind_\Sigma^\Gamma(\sigma)$ is a direct sum of finitely many factorial  representations of type II. At least one of them is tracial\footnote{In fact all of them, but we don't need this}.
\item Take $\sigma'$ another representation of $\Sigma$ and $\rho'$ denotes its induced $\Gamma$-representation. If a subrepresentation of $\rho$ is weakly contained in $\rho'$ then $\sigma \circ \Ad(g)$ is weakly contained in $\sigma'$, for some $g \in \Lambda$ (and $\Ad(g)$ denotes the automorphism of $\Sigma$ obtained by $g$ conjugation).
\end{enumerate}
\end{prop}
\begin{proof}
(1) In fact we we can give a more precise statement using the induction by stages principle. Denote by $\pi := \Ind_\Sigma^{\Lambda}(\sigma)$, so that $\rho$ is conjugate to the induced representation of $\pi$. Then by Lemma \ref{going up}, we know that $\pi$ is a direct sum of finitely many factorial representations of type II$_1$. We will prove that $\rho(\Gamma)'' = B(\ell^2(\Gamma/\Lambda)) \ovt \pi(\Lambda)''$ and that $C^*(\rho(\Gamma))$ contains a non-zero element $x$ such that $(p_{\Lambda} \ot 1)x = x$, proving the traceability property.

Denote by $H_0$ the Hilbert space on which $\pi$ acts. Denote by $I := (\Gamma/\Lambda) \setminus \Lambda$. Then $\ell^2(\Gamma/\Lambda) \ot H_0$ is the direct sum of $H_1 := \delta_\Lambda \ot H_0$ and $H_2 := \ell^2(I) \ot H_0$, both of which are $\Lambda$-invariant. Denote by $\pi_1$ and $\pi_2$ the representations of $\Lambda$ obtained this way.

The result will follow exactly as in Proposition \ref{main prop} once we prove that no sub-representation of $\pi_1$ is weakly contained in $\pi_2$. Let us restrict further the discussion to $\Sigma$.

The restriction $\pi_1|_\Sigma$ is the direct sum of finitely many representations of the form $\sigma \circ \Ad(g)$ for elements $g \in \Lambda$. Indeed this follows by the proof of Lemma \ref{going up}. In particular each such representation is factorial. Note that $\Ad(g)$ permutes the groups $\Sigma_k$, $k = 1,\dots,d$. It follows that the restriction of these representations $\sigma \circ \Ad(g)$ to $\Sigma_k$ is factorial and not weakly contained in the regular representation for every $k = 1,\dots,d$.

{\bf Claim.} No subrepresentation of $\pi_1$ is weakly contained in $\pi_2$.

Assume the contrary. Then by restricting to $\Sigma$, we find that some $\Sigma$-subrepresentation of $\pi_1|_\Sigma$ is weakly contained in $\pi_2|_\Sigma$. By Lemma \ref{factorial}, we find that $\sigma \circ \Ad(g)$ is weakly contained in $\pi_2|_\Sigma$ for some $g \in \Lambda$. For simplicity denote by $\sigma' := \sigma \circ \Ad(g)$ and $H'$ the Hilbert space on which it acts. 

Denote by $\Phi: C^*(\pi_2(\Sigma)) \to C^*(\sigma'(\Sigma))$ the C*-morphism such that $\Phi(\pi_2(h)) = \pi'(h)$, for all $h \in \Sigma$. Use Arveson extension theorem to extend $\Phi$ to a ucp map $E: B(\ell^2(I) \ot H_0) \to B(H')$. 
For every $k$, denote by $I_k \subset I$ the set of cosets $g\Lambda \in I$ such that $\Sigma_k \cap g\Lambda g^{-1}$ is amenable. By Lemma \ref{a-normal family}, we have $\bigcup_{k=1}^d I_k = I$. 
Denote by $p_k \in B(\ell^2(I) \ot H_0)$ the orthogonal projection onto $\ell^2(I_k) \ot H_0$. Since $I_k$ is $\Sigma_k$-invariant, $p_k$ commutes with $\pi_2(\Sigma_k)$. Denote by $r_k \in B(H')$ the support projection of $E(p_k)$. 

Since $\bigcup_{k = 1}^d I_k = I$, we get $\sum_k p_k \geq 1$, and thus we may find $k$ such that $E(p_k) \neq 0$. In this case $r_k$ is a non-zero projection invariant under $\sigma'(\Sigma_k)$. Moreover, the map $E$ restricted to $p_kB(\ell^2(I) \ot H_0)p_k$ witnesses that $p_k\pi_2|_{\Sigma_k}$ weakly contains $r_k\sigma'|_{\Sigma_k}$. Since $\sigma'|_{\Sigma_k}$ is factorial, Lemma \ref{factorial} implies that $r_k\sigma'|_{\Sigma_k}$ is weakly equivalent to $\sigma'|_{\Sigma_k}$. Moreover, our choice of $I_k$ and $p_k$ and Lemma \ref{Mackey} imply that $p_k\pi_2|_{\Sigma_k}$ is weakly contained in the regular representation of $\Sigma_k$. So we arrive at the conclusion that $\sigma'|_{\Sigma_k}$ is weakly contained in the regular representation. 
But we observed that this was impossible. This contradiction finishes the proof of the claim, and the rest of (1) follows as in the proof of Proposition \ref{main prop}.

(2) Assume that a subrepresentation of $\rho$ is weakly contained in $\rho'$. By (1), we know that $\rho$ is the direct sum of finitely many factorial representations of the form $\Ind_\Lambda^\Gamma(\pi_0)$ for some factorial representation $\pi_0$ of $\Lambda$. In fact $\pi_0$ is a factorial subrepresentation of $\pi = \Ind_\Sigma^\Lambda(\sigma)$.

By Lemma \ref{factorial}, we deduce that one such factorial summand $\Ind_\Lambda^\Gamma(\pi_0)$ is weakly contained in $\rho'$. Restricting to $\Sigma$ we deduce that in particular $\pi_0|_\Sigma$ is weakly contained in $\rho'|_\Sigma$. But $\pi_0|_\Sigma$ is a subrepresentation of $\pi|_\Sigma$, which is the direct sum of finitely many factorial representations of the form $\sigma \circ \Ad(g)$. So applying again Lemma \ref{factorial}, we find that some $\sigma \circ \Ad(g)$ is weakly contained in $\rho'|_\Sigma$.

Now we apply the analysis made in (1) to $\rho'|_\Sigma$. It is the direct sum of finitely many representations $\sigma' \circ \Ad(h)$, for some elements $h \in \Lambda$, together with one representation $\sigma_2$, which is the restriction to $\Sigma$ of the representation on $\ell^2(I) \ot H_0'$. The proof of the claim above shows that $\sigma \circ \Ad(g)$ is not weakly contained in $\sigma_2$. By Lemma \ref{cut}, we then deduce that $\sigma \circ \Ad(g)$ is weakly contained in some $\sigma' \circ \Ad(h)$ for some $h \in \Lambda$, and we get the desired conclusion.
\end{proof}

\begin{proof}[Proof of Theorem \ref{cor2}]
By Lemma \ref{many rep}, we may find an uncountable family $\pi_i$, $i \in I$, of factorial representations of type II$_1$ of $\SL_2(\Z)$ which are all amenable and such that $\pi_i$ is not weakly contained in $\pi_j \circ \alpha$ for any $i \neq j$ and $\alpha \in \Aut(\SL_2(\Z))$. Since $I$ is uncountable, we may cut it in $d$ copies of itself, $I \simeq I^d$. So in fact, we may find uncountably many $d$-tuples of representations $(\pi_i^1, \dots,\pi_i^d)$, $i \in I$, such that $\pi_i^k$ is not weakly contained in $\pi_j^\ell$ for every $(i,k) \neq (j,\ell)$, and $\alpha \in \Aut(\SL_2(\Z))$.

Each such tuple gives a representation $\sigma_i$ of $\Sigma$, defined by
\[\sigma_i(g_1,\dots,g_d) = \pi_i^1(g_1) \ot \cdots \ot \pi_i^d(g_d), \text{ for every } (g_1,\dots, g_d) \in \Sigma.\]
Since the action of $\Lambda$ on $\Sigma$ permutes the factors $\Sigma_i$, we find that $\sigma_i$ is not weakly contained in $\sigma_j \circ \Ad(g)$ for every $i \neq j$ in $I$ and $g \in \Lambda$.

Moreover, each $\sigma_i$ is factorial since $\sigma_i(\Sigma)'' = \pi_i^1(\Sigma_1)'' \ovt \cdots \ovt \pi_i^d(\Sigma_d)''$. Likewise, its restriction to $\Sigma_k$, $k = 1,\dots, d$, is factorial and amenable (hence not weakly contained in the regular representation.

We may apply Proposition \ref{main prop 2} to find in the induced $\Gamma$-representation of each $\sigma_i$ a direct summand $\rho_i$ which is factorial and traceable. The family $\rho_i$ is uncountable and and none of them is weakly contained in any other.
\end{proof}


\bibliographystyle{plain}

\end{document}